\documentclass[journal]{IEEEtran}
\usepackage{cite}
\usepackage{amssymb}
\usepackage{amsthm}
\usepackage{stmaryrd}
\usepackage[nosumlimits]{amsmath}
\usepackage{algorithmic}
\usepackage{graphicx}
\usepackage{textcomp}
\usepackage{setspace}
\usepackage{cleveref}
\usepackage{xcolor}
\usepackage{steinmetz}
\usepackage{mathtools}
\usepackage{url}

\usepackage[labelformat=simple]{subcaption}

\usepackage{accents}

\usepackage{algorithm}

\newtheorem{thm}{Theorem}
\newtheorem{prop}{Proposition}
\newtheorem{definition}{Definition}
\newtheorem{rem}{Remark}
\newtheorem{lem}{Lemma}

\newtheorem{assum}{Assumption}

\title{Ensuring Transient Stability with Guaranteed Region of Attraction in DC Microgrids}
\author{Jianzhe~Liu,~\IEEEmembership{Member,~IEEE}, Yichen~Zhang,~\IEEEmembership{Senior Member,~IEEE},\\
Antonio~J.~Conejo,~\IEEEmembership{Fellow,~IEEE},
Feng~Qiu,~\IEEEmembership{Senior Member,~IEEE},
\thanks{J. Liu, Y. Zhang, and F. Qiu are with the Energy Systems Division, Argonne National Laboratory, Lemont, IL 60439, USA. Email: \{jianzhe.liu, yichen.zhang, fqiu\}@anl.gov. Argonne National Laboratory's work is based upon work supported by the U.S. Department of Energy’s Office of Energy Efficiency and Renewable Energy (EERE) under the Water Power Technologies Office Seedling Project Award Number 37893.}
\thanks{A. J. Conejo is with the Department of Integrated Systems Engineering and the Department of Electrical and Computer Engineering, The Ohio State University, Columbus, OH, USA. Email: conejo.1@osu.edu. }
}
\allowdisplaybreaks[4]
\begin{document}

\maketitle
\begin{abstract}
    DC microgrids have promising applications in renewable integration due to their better energy efficiency when connecting DC components. However, they might be unstable since many loads in a DC microgrid are regulated as constant power loads (CPLs) that have a destabilizing negative impedance effect. As a result, the state trajectory displacement caused by abrupt load changes or contingencies can easily lead to instability. Many existing works have been devoted to studying the region of attraction (ROA) of a DC microgrid, in which the system is guaranteed to be asymptotically stable. Nevertheless, existing work either focuses on using numerical methods for ROA approximations that generally have no performance guarantees or cannot ensure a desired ROA for a general DC microgrid. To close this gap, this paper develops an innovative control synthesis algorithm to make a general DC microgrid have a theoretically guaranteed ROA, for example, to cover the entirety of its operating range regarding state trajectories. We first study the nonlinear dynamics of a DC microgrid to derive a novel transient stability condition to rigorously certify whether a given operating range is a subset of the ROA; then, we formulate a control synthesis optimization problem to guarantee the condition's satisfaction. This condition is a linear constraint, and the optimization problem resembles an optimal power flow problem and has a good computational behavior. Simulation case studies verify the validity of the proposed work.
\end{abstract}


\section{Introduction}

In shifting into a zero-carbon society, the next decades will witness an accelerated integration of distributed energy resources (DERs). Many DERs, such as PV units, plug-in electric vehicles, and energy storage systems, operate in direct current (DC)~\cite{justo2013ac}. For better energy efficiency, a DC microgrid is an appealing network configuration to interconnect these components in a low- or medium-voltage DC system~\cite{dragivcevic2015dc}. 

Despite the prospective benefits, a DC microgrid is subject to potential instability risks caused by constant power loads (CPLs)~\cite{liu2018robust}. CPLs are nonlinear components that introduce a negative impedance effect that weakens the system damping and affects the stability margin~\cite{li2019large}. 

To tackle the issue, both local stability (small-signal stability) and transient stability (large-signal stability) problems for DC microgrids have been studied in the literature. 
The former focuses on the system dynamical behaviors in a sufficiently small region surrounding an equilibrium. Linearization-based techniques are usually applied to develop the stability results: If the linearized system is stable at the equilibrium, then the nonlinear DC microgrid has a non-empty region of attraction (ROA) in the vicinity of the equilibrium in which asymptotic stability is ensured~\cite{slotine1991applied,isidori1995nonlinear}. Under such notion, one can then claim that after ``small'' disturbances a locally stable DC microgrid can return to equilibria~\cite{liu2018robust,tahim2015model,bara2016on}.

A variety of local stability conditions and stabilization controllers have been developed. In~\cite{middle1976input}, a frequency-domain model is used to develop a Nyquist stability criterion; in~\cite{tahim2015model}, the authors study the state-space model to develop stability conditions based on eigenvalue analysis of the Jacobian matrix; in~\cite{liu2018robust}, stability conditions based on linear matrix inequality (LMI) constraints are developed to conduct robust stability analysis against load uncertainty. 
Based on these results, various DC microgrid controllers have been developed for local stabilization~\cite{dragivcevic2015dc}, including primary droop control~\cite{josep2011hierar}, secondary distributed control~\cite{morstyn2015unified}, and tertiary control for economic dispatch~\cite{maulik2018stability}. 

Although local stability results can guarantee the existence of a ROA, they are usually unable to explicitly characterize its geometry. 
Consequently, they have limitations in determining whether the system can still retain asymptotic stability when significantly deviated from the equilibrium to a given point in the state space. 
This knowledge, however, is critical for understanding the system stability margin and ensuring system stability after ``large'' disturbances~\cite{marx2011large}. 

To deal with this issue, transient stability problems for DC microgrids have been studied. 
Existing works have developed many stability criteria for the nonlinear system, including those based on Lyapunov's direct method~\cite{chang2020large,herrera2015stability,babaiahgari2020dynamic,xu2019robust}, Brayton-Moser's mixed potential method~\cite{marx2011large,jiang2019conservatism}, bifurcation theory~\cite{tahim2014modeling}, and singular perturbation theory~\cite{gui2021large}. 
First, stability criteria have been developed to certify whether a given operating set is a subset of the ROA, then a ROA approximation can be obtained by iteratively applying stability analysis results to an enlarging candidate set. 
In~\cite{herrera2015stability,babaiahgari2020dynamic}, iterative algorithms are developed to apply Lyapunov stability conditions in each step; if the conditions are satisfied, a candidate set is enlarged to enter the next step; otherwise, an estimation of the ROA is considered final. 
Second, stability conditions have been applied to control design as well~\cite{gui2021large,zhang2019finite,chang2020large,xu2019robust}. 
These designs usually apply stability conditions to the closed-loop system to obtain transient performance guarantees~\cite{gui2021large,xu2019robust}.
For instance, in~\cite{zhang2019finite}, a decentralized linear controller is developed, and through Lyapunov stability analysis it is shown that when the control parameters are sufficiently large the system ROA must be of the desired size;
in~\cite{chang2020large}, the mixed potential method is applied to provide conditions on the system parameters to ensure transient stability.

Despite the rich literature on DC microgrid stability analysis and control, existing works have the following limitations. First, local stability results generally cannot be directly extended to transient stability problems due to the linearized dynamics~\cite{tahim2015model,bara2016on,lu2015stability}. Second, although some works have tackled the nonlinearity and developed ROA estimation methods, they generally rely on iterative algorithms that have no performance guarantee over the estimated ROA~\cite{herrera2015stability,babaiahgari2020dynamic}. Consequently, they may not be able to provide insights into critical operational problems like whether or not a DC microgrid is stable for the entire operating range. Third, existing stabilization control works have developed conditions to certify a ROA, but a systematic control synthesis approach is still needed to pinpoint the control parameters that modify the geometry of a ROA to contain a certain operating set~\cite{zhang2019finite,chang2020large} or achieve satisfactory transient and steady-state performance~\cite{gui2021large,xu2019robust}. Fourth, most existing works either focus on a single converter~\cite{gui2021large} or the specific DC microgrid configuration with only one common DC bus~\cite{zhang2019finite,chang2020large,xu2019robust}. There are insufficient studies on DC microgrids with general topology to account for a wider range of applications. Fifth, although some existing works have discussed the impacts of uncertainties in system parameters on transient stability~\cite{xu2019robust,gui2021large,chang2020large}, the control methods to withstand uncertainties in loading conditions are not adequately studied. 

This paper closes these research gaps by developing a systematic control synthesis algorithm to ensure that a general DC microgrid has a theoretically guaranteed ROA. To do so, we first transform the nonlinear dynamical model of the DC microgrid into an equivalent form that resembles a linear parameter varying (LPV) system if the nonlinear function is considered as a superimposed unknown parameter. Leveraging its special structure, we derive a stability condition on the nonlinearity to certify that a given set is contained in the system ROA. This condition is a linear constraint with respect to system equilibria. Such simple structure makes it amenable for control synthesis. We formulate an optimization problem reminiscent of an optimal power flow (OPF) problem to design the system equilibria accordingly. Owing to the problem structure, existing OPF results, such as suitable solvers, robust formulations to account for load uncertainties, and convex relaxation formulations, can be applied to enhance the computational performance. Note that security constrained OPF problems with stability constraints have been studied for AC power systems~\cite{zarate2009securing,zarate2010opf,capitanescu2011state}. These works either rely on time-domain simulations to derive the stability margins or develop stability constraints based on the one-machine-infinite-bus model~\cite{capitanescu2011state,abhyankar2017solution}. Therefore, they might not directly apply to a DC microgrid with general topology to ensure a certain ROA.

Considering the above literature review, the contributions of this paper include:
\begin{itemize}
    \item We develop a novel control synthesis approach to ensure that a general DC microgrid has a guaranteed ROA. As a result, the method can be applied to ensure system transient stability in the entirety of the operating range.
    \item We develop DC microgrid stability conditions that reveal the relationship between system equilibria and transient stability.
    \item The proposed optimization problem for control parameter design resembles a classic OPF problem, which generally has a good computational behavior. We show as well that under mild conditions, this optimization problem can be transformed into a convex conic optimization problem to achieve better computational efficiency.
    \item The proposed method can be employed for DER dispatching or re-dispatching to ensure economic operations while guaranteeing transient stability.  
\end{itemize}

The rest of the paper is organized as follows: Section~\ref{sec:pre} introduces the notation and some fundamental concepts for DC microgrids; Section~\ref{sec:ps} states the problem; Section~\ref{sec:tran_stb} presents the main results; Section~\ref{sec:simu} presents simulation results to show the validity of the proposed work; and at last, Section~\ref{sec:conclu} concludes the paper and discusses future research directions.

\section{Preliminary}\label{sec:pre}
This section summarizes the notation of the paper and reviews the fundamentals of DC microgrids.
\subsection{Notation}
For a vector $v$, let $v_k$ represent its $k$-th entry. We say that square matrix $A$ is Hurwitz if the real parts of all its eigenvalues are negative. In the following we provide a list of the operators used in the papers:
\begin{itemize}
    \item \textbf{$[\cdot]$} yields a diagonal matrix with the vector's entries being the diagonal entries;
    \item $\varoslash$ yields the element-by-element division between two vectors;
    \item $\odot$ represents the element-wise product of two vectors;
    \item $\sup$ and $\inf$ find the element-wise supremum and infimum in the arguments;
    \item $\sqrt{\cdot}$ finds the element-wise square root of the vector in the argument.
    \item $\det\_{\mathrm{rootn}}(\cdot)$ yields the $n$-th root of the determinant of a semi-positive definite matrix
\end{itemize}


\subsection{DC Microgrid Model}\label{sec:pre:mdl}
In this paper we study a DC microgrid with $n_{\text{b}}$ buses and $n_{\text{t}}$ power lines as shown in Fig.~\ref{fig:dc_bus}. For the $n_{\text{b}}$ buses, there are $n_\text{s}$ source buses and $n_\ell$ load buses. 
In addition, let $n=n_{\text{b}}+n_\text{t}$. 
The network is represented by a resistance--inductance--capacitance (RLC) circuit, as shown in Fig.~\ref{fig:dc_inputs}. On each DC bus, let there be a composite shunt capacitor and resistor, and let each power line be composed of a line resistor and inductor. 



We consider that the loads are constant power loads (CPLs). For the $j$-th CPL, let its power output be $p_{\ell j} \leq 0$.  
In addition, we consider that each source bus is connected to a voltage source. For the $k$-th voltage source, let the voltage reference be $u_k$ and let $R_{\text{s}k}$ represent the composite terminal resistance, combining the virtual and physical resistance. 
Let $p_{\ell j}$ be a non-positive constant for simplicity. 

\begin{figure}[!ht]
\begin{subfigure}{.40\textwidth}
    \centering
    \includegraphics[width=0.7\textwidth]{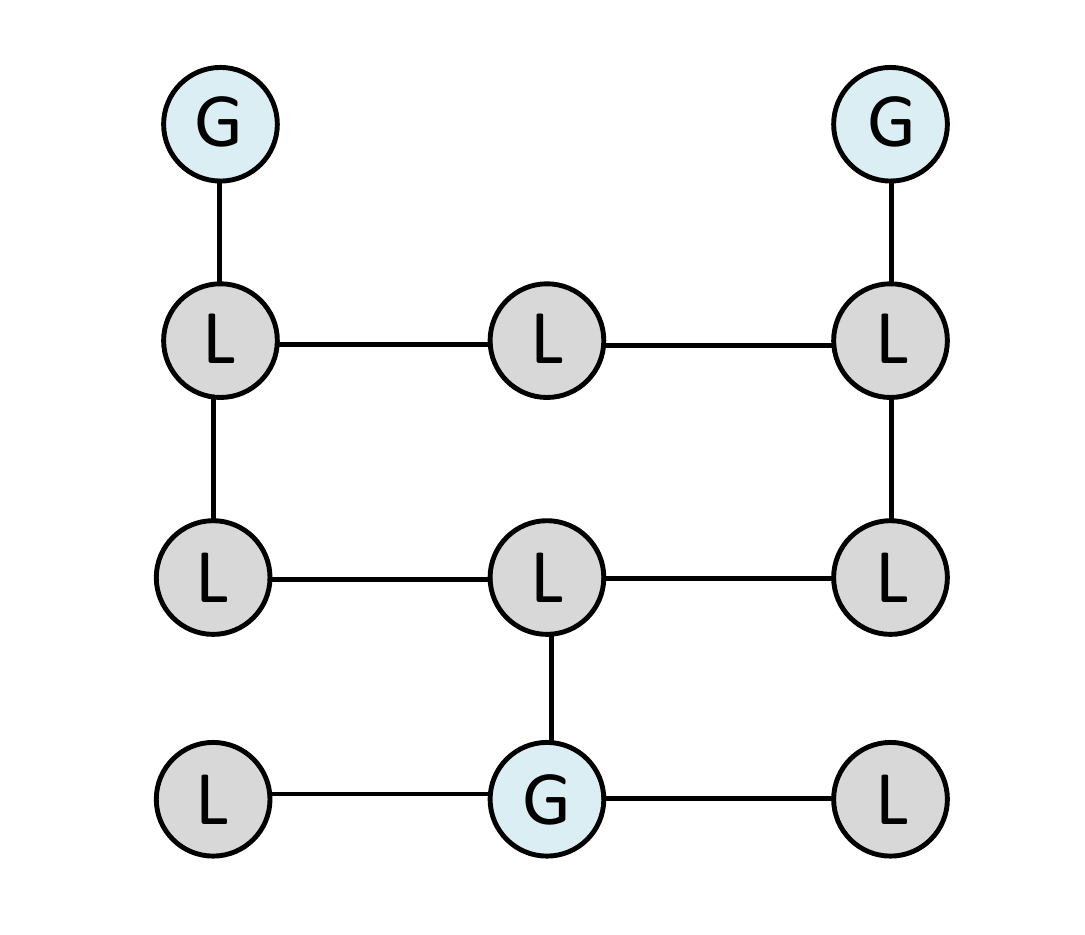}
    \caption{Example DC microgrid topology}\label{fig:dc_bus}
\end{subfigure}

\begin{subfigure}{.40\textwidth}
    \centering
    \includegraphics[width=0.8\textwidth]{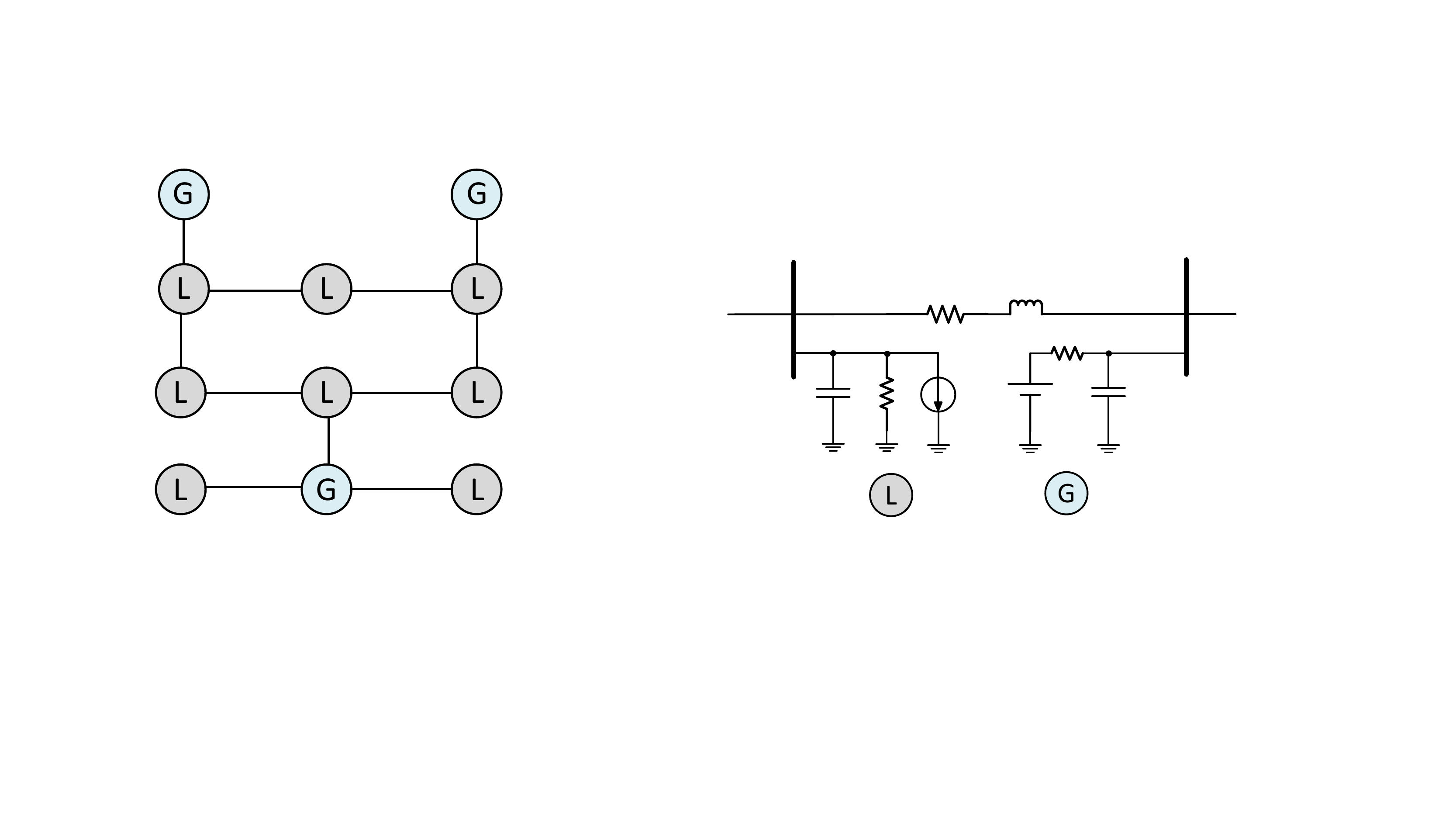}
    \caption{Detailed RLC circuit}\label{fig:dc_inputs}
\end{subfigure}
\caption{General DC microgrid topology and circuit model\label{fig:dc_topo}}
\end{figure}


\subsubsection{Dynamical Model}
Let $i \in \mathbb{R}^{n_{\text{t}}}$ be the line currents, and $v_{\text{s}} \in \mathbb{R}^{n_\text{s}}$ and $v_{\ell} \in \mathbb{R}^{n_\ell}$ be DC bus voltage of the source and load buses, respectively. Let $x \triangleq [i^\top,v_\text{s}^\top,v_\ell^\top]^\top$ be the state variable vector; $p_\ell \in \mathbb{R}^{n_\ell}$ is a vector of CPL power; and $u \in \mathbb{R}^{n^{\text{s}}}$ represents the voltage set-point of the sources.

The dynamics of a DC microgrid with CPLs are described by the following nonlinear dynamic system~\cite{liu2019optimal},
\begin{equation}\label{eq:mdl_main}
    D\dot{x}(t) = Ax(t) + B_2u\varoslash R_s + B_1 \left(p_\ell \varoslash C_1x(t) \right).
\end{equation}
where $D$ is a diagonal matrix whose diagonal entries are line inductances and bus capacitances; $A$ is composed by the addition of a negative diagonal sub-matrix and a skew-symmetric sub-matrix, the diagonal sub-matrix is determined by the system resistance; $B_1$ and $B_2$ are input matrices, which can be used to extract the source and load voltage $v_s = B_2^\top x$ and $v_\ell = B_1^\top x$. For simplicity, we let $C_1 = B_1^\top$ and $C_2 = B_2^\top$ in the rest of the paper. The last term on the right hand side of equation~\eqref{eq:mdl_main} is an additive nonlinear function contributed by the CPLs. We exemplify these system matrix expressions in Appendix~\ref{app:mdl}. It is worth mentioning that most DC microgrid models studied in the literature~\cite{dragivcevic2015dc,morstyn2015unified,herrera2015stability,liu2019optimal} conform to the general structure of~\eqref{eq:mdl_main} whose key feature is that the system dynamics have additive linear and nonlinear components. Our results can be easily applied to these models regardless of how the system matrices are parameterized.

\subsubsection{Steady-State Model}
The steady-state model describes the operating points and the power flows of a DC microgrid. 
Let $x^e$ be the system steady-state operating point, $v^e_\ell$ be the load bus voltage at such operating point, and $p_s$ be the vector of the power injections from the source buses. With the above definitions, the power flow equations are given by
\begin{equation}\label{eq:mdl_pf}
        p_s = \left[u\right]\left(Y_{\text{s}\ell}v^e_\ell +Y_{\text{ss}}u\right),
        \quad p_\ell = \left[ v^e_\ell\right] \left( Y_{\ell\ell}v^e_\ell +Y_{\ell\text{s}}u\right),
\end{equation}
where $Y_{\text{ss}}$, $Y_{\text{s}\ell}$, and $Y_{\ell\ell}$ are the conductance matrices describing the source-source, source-load, and load-load sub-networks, and $Y_{\text{s}\ell} = Y^\top_{\ell\text{s}}$. 

The solution of the power flow equations is the operating point of the system. It can be shown that the operating points are equivalent to the equilibria of~\eqref{eq:mdl_main}~\cite{liu2018existence}. In the rest of the paper, we will refer to the system equilibria as the operating points. 

\subsection{OPF}\label{sec:pre:opf}
To ensure an optimal economic operation, an optimal power flow (OPF) problem is usually solved to find an optimal operating point~\cite{gan2014optimal}.
An OPF problem is a constrained optimization problem aiming to achieve minimum generation fuel costs, minimum power losses on the lines, flattened voltage profiles, and others~\cite{dommel1968optimal}, 
by varying the voltage source setpoints and power injections as stated in~\eqref{eq:mdl_pf}.

An OPF is subject to the power flow equations and other operational constraints. Let $\mathcal{U}$, $\mathcal{X}^e$, and $\mathcal{P}_{\text{s}}$ be the operational constraint sets for $u$, $x^e$, and $p_\text{s}$, respectively. Specifically, the operating point $x^e$ has interval constraints given by $x^e_- \leq x^e \leq x^e_+$, where $x^e_-$ and $x^e_+$ are some lower and upper bounds representing the operational or physical constraints for the operating point.

An example OPF problem for generation cost reduction is formulated as follows,
\begin{subequations}\label{eq:prob_opf}
\begin{align}
     \text{\textbf{OPF:}}& && \min_{u,x^{\text{e}},v^\text{e}_\ell,p_\text{s}} c^\top p_s,\label{eq:opf_cf}\\
    \text{s. t.} & && \eqref{eq:mdl_pf}, \; v^\text{e}_\ell = C_1x^{\text{e}}, \label{eq:prob_opf:con1}\\
    & && u \in \mathcal{U},\; x^e \in \mathcal{X}^e,\; p_\text{s} \in \mathcal{P}_\text{s},\label{eq:prob_opf:con2}
\end{align}
\end{subequations}
where $c \in \mathbb{R}^{n_{\text{s}}}_+$ is a vector of generation cost coefficient. 

Extensive studies have been carried out to solve OPF problems like~\eqref{eq:prob_opf}. Recent advancements, such as OPF exact relaxation methods~\cite{jabr2011exploiting,molzahn2018survey}, have further shown that the global optimum of some large-scale networks can be found in polynomial time. Regarding DC networks, existing work~\cite{gan2014optimal, li2018optimal} has developed exact conic relaxation methods that have promising computational characteristics.

\subsection{Region of Attraction} \label{sec:pre:roa}
Aside from steady-state performance, transient stability is a key issue for DC microgrid operations. Most research works seek to determine a region of attraction (ROA) to ensure the convergence of system trajectories after a disturbance~\cite{slotine1991applied}.

Although the ROA of a nonlinear system like~\eqref{eq:mdl_main} is usually difficult to characterize, Lyapunov stability theory provides an effective method to use sub-level sets to inner approximate a ROA~\cite{herrera2015stability}. Let $V(x)$ be a continuously differentiable candidate Lyapunov function for a nonlinear system like~\eqref{eq:mdl_main}. A given set is a subset of the ROA if:
\begin{itemize}
    \item[a)] The set is contained in a sub-level set defined by $\{x:V(x)\leq \alpha\}$ for some positive scalar $\alpha$;
    \item[b)] $V(x)$ is positive definite within the sub-level set except at a unique equilibrium, in which it is zero; and
    \item[c)] the time derivative of $V(x)$ is negative definite for any state trajectory inside of the sub-level set except at the unique equilibrium where the derivative is zero.
\end{itemize}
Note that these conditions are sufficient conditions to indicate that 1) the given set is a subset of a positively invariant set such that if the system trajectory enters the latter at a given time, it will remain staying in it for the future time; and 2) the system trajectory will then asymptotically converge to a unique equilibrium.

\section{Problem Statement}\label{sec:ps}
The DC microgrid operation is subject to various disturbances. For example, a DC microgrid usually has a significant penetration of intermittent renewable resources and variable loads that may induce significant changes in the loading condition. 

With such disturbances, a DC microgrid might suddenly drift away from a given operating point. If the DC microgrid is still within the ROA of the post-disturbance system, it can safely return to the original operating point. It is thus desirable to ensure that the entire operating range of the microgrid is a subset of the ROA. 

Let $\mathcal{X}^\text{o} \triangleq [x^\text{o}_-,x^\text{o}_+]$ be the operating range, where $x^\text{o}_- \triangleq x^e - \Delta \bar{x}$ and $x^\text{o}_+ \triangleq x^e +  \Delta \bar{x}$ with $\Delta \bar{x} > 0$ being given bounds for the trajectory variations. To determine whether or not the post-disturbance system is transiently stable for the entire operating range, we need to certify that after the time instant that the system enters $\mathcal{X}^\text{o}$, it must asymptotically converge to $x^e$.

It is worth mentioning that the bounds for the trajectory variation, $\Delta \bar{x}$, is an arbitrarily chosen positive vector. For example, it can in fact be chosen to cover the entirety of the operating range or beyond to obtain some security margin.

This paper seeks to design the control $u$ for a DC microgrid described by~\eqref{eq:mdl_main} to ensure \textit{transient stability with a guaranteed ROA} under the definition below.

\begin{definition}~\label{def:tran_stab}
A DC microgrid governed by~\eqref{eq:mdl_main} is called transiently stable with a guaranteed ROA if the system ROA contains $\mathcal{X}^\text{o}$.
\end{definition}

In light of the above Definition~\ref{def:tran_stab}, let time $t=0$ represent either the moment that a fault is cleared or the moment that the system control is re-dispatched in response to an event. The system trajectory (with $t>0$) describes the post-disturbance evolution.

The difficulties of the problem described above include: The ROA of a nonlinear system like~\eqref{eq:mdl_main} is difficult to characterize, and there is insufficient results on developing a systematic control synthesis approach to guarantee a given size of the ROA to cover a desired set. Moreover, checking the conditions described in Section~\ref{sec:pre:roa} is difficult for our problem in that set $\mathcal{X}^{\text{o}}$ is dependent on the design variable $u$. Specifically, the system operating point $x^e$ is a function of $u$, as discussed in Section~\ref{sec:pre:mdl}. If $u$ is a design variable to be determined, $\mathcal{X}^{\text{o}}$ is not known a priori to enable a direct application of the existing ROA estimation works. Therefore, a new stability condition needs to be developed under this circumstance. In addition, although heuristic methods or iterative algorithms can be applied to vary $u$ to change the ROA, they may entail considerable computational burden and are generally not suitable to provide theoretical performance guarantees.

In the following section, we develop a novel stability condition and a scalable systematic control synthesis approach to solve this problem.

\begin{rem}
The operating range $\mathcal{X}^\text{o}$ is a hypercube. Practical experience indicates that the power system operating range is usually defined in such form. For example, the power system frequency operating range of the U.S. is usually set as 60$\pm 5\%$ Hz~\cite{grid101}, and the National Electrical Manufacturers Association recommends that motors to have a voltage variation range of $\pm$10\% of the rated voltage~\cite{voltbound}.
\end{rem}


\section{Transient Stability Analysis and Control}\label{sec:tran_stb}
The proposed approach to ensure transient stability with a guaranteed ROA contains three steps: 1) We re-arrange system~\eqref{eq:mdl_main} into an equivalent homogeneous-like system where the nonlinear component can be superimposed by a parameter uncertainty; 2) leveraging this special structure, we derive a bound on the parameter uncertainty to make the homogeneous-like system stable; 3) we formulate an optimization problem to design a control that ensures that the nonlinear component always satisfies the derived bound when $x(0) \in \mathcal{X}^\text{o}$. We show that by doing so the system must be transiently stable with a guaranteed ROA.

We derive the stability condition under the following assumption:
\begin{assum}\label{assum:positive}
The initial load bus voltages are positive, i.e. \mbox{$C_1x(0) > 0$}.
\end{assum}
This is a mild condition in line with physical facts.
It is worth mentioning that we no longer need to make this assumption later into the control design process.

\subsection{Homogeneous-Like System}\label{sec:tran_stb:homo}
As discussed in Section~\ref{sec:ps}, tackling the system nonlinearity is one key difficulty of the problem. We first re-arrange model~\eqref{eq:mdl_main} into a homogeneous-like form in order to make it more amenable to our analysis.



Let the difference between system state and an equilibrium $x^e$ be $\Delta x \triangleq x - x^e$. With this change of coordinates, system~\eqref{eq:mdl_main} becomes:
\begin{equation}\label{eq:mdl_small_ini}
    D\Delta \dot{x}(t) = A\Delta x(t) + B_1\left[h(x^e,\Delta x(t))\right]C_1\Delta x(t),
\end{equation}
where $h(x^e,\Delta x) \triangleq -p_\ell \varoslash  \left(C_1(x^e + \Delta x)\odot C_1x^e\right)$. In the following, we denote $h(x^e,\Delta x)$ by $h$ for brevity. The term $B_1[h]C_1$ is still a nonlinear function, and it is right-multiplied by $\Delta x(t)$. System~\eqref{eq:mdl_small_ini} is then recast into the following homogeneous-like form when $h$ is treated as a superimposed parameter uncertainty:
\begin{align}\label{eq:mdl_small_sec}
    \Delta \dot{x}(t) =  \hat{A}(h)\Delta x(t), 
\end{align}
where $\hat{A}(h) = D^{-1}\left( A + B_1[h]C_1 \right)$. 
This special structure leads to meaningful implications. System~\eqref{eq:mdl_small_sec} is now a homogeneous linear parameter varying (LPV) system. For such systems, we can derive an interval bound on the parameter uncertainties to ensure stability. We show that the stability analysis assists us to develop stability conditions for the original nonlinear system.

Note that with the re-arrangement carried out, the operating point is shifted to the fixed point origin. Then, the ROA that we need to guarantee becomes
\begin{equation}
    \Delta \mathcal{X}^{\text{o}} \triangleq \{\Delta x \in \mathbb{R}^n:\Delta x \in [-\Delta \bar{x},\Delta \bar{x}]\},\nonumber
\end{equation}
which is a well-defined polytope. The transient stability problem that we are now concerned with is to ensure that $\Delta \mathcal{X}^\text{o}$ is contained in the ROA of~\eqref{eq:mdl_small_ini}.



\subsection{LPV Condition}\label{sec:tran_stb:lpv}
We first find a stability condition over $h$, taking into account that~\eqref{eq:mdl_small_sec} is treated as an LPV with its system matrix parameterized by the superposition parameter $h$. Then, we transform this condition into a transient stability condition for our nonlinear model. 

Recall that for a homogeneous linear system, the existence of a quadratic Lyapunov function is a necessary and sufficient stability condition. This allows us to employ a candidate quadratic Lyapunov function. Let $V(\Delta x(t)) \triangleq \Delta x(t)^\top P \Delta x(t)$, where $P \succ 0$ is a positive definite matrix. 

An interval constraint on $h$ can be found to keep the Hurwitz stable LPV system matrix. Finding an interval constraint is motivated by the following observation. Under Assumption~\ref{assum:positive} and the constraint that $x^e_- \leq x^e \leq x^e_+$, $h$ is bounded from both above and below: 
\begin{subequations}
\begin{align*}
 &h \geq -p_\ell \varoslash \left( C_1\left(x^e_+ +\Delta \bar{x} \right)\odot C_1x^e_+ \right)  \geq 0,\\
 &h \leq -p_\ell \varoslash \left( C_1\left(x^e_- -\Delta \bar{x} \right)\odot C_1x^e_- \right).
\end{align*}
\end{subequations}

From the above discussion, we seek to find the following interval set of $h$, $\mathcal{H} \triangleq \{ h: h \in [0,\beta h_{+}^{\text{o}}] \}$, to keep the system matrix stable, where $h_{+}^{\text{o}} > 0$ is an initial guess for the upper bound of $h$ and $\beta > 0$ is a scaling factor to vary the size of $\mathcal{H}$. The largest $\mathcal{H}$ that makes Hurwitz stable matrix $\hat{A}$ can be arbitrarily approximated by the solution of the following generalized eigenvalue problem (GEVP):
\begin{subequations}\label{eq:prob_gevp}
\begin{alignat}{2}
    &\textbf{GEVP:}\qquad \max_{P \succ 0,\hat{P},\beta > 0,\tau > 0} \beta,\quad \text{s. t.}\\
    &\left[ 
    \begin{array}{cc}
        \hat{P}A\!+\!A^\top \hat{P}\!+\!P &  \hat{P}B_1\!+\!\frac{1}{2}\beta \tau C_1^\top [h_{+}^{\text{o}}] \\
        B_1^\top \hat{P}\! + \! \frac{1}{2}\beta \tau[h_{+}^{\text{o}}]C_1 & -\tau I
    \end{array}
    \right] \!\! \preceq \!\! 0,\label{eq:cond23_lyap}\\
    &\hat{P} = PD^{-1}, \label{eq:cond23_misc}
\end{alignat}
\end{subequations}
where $\tau \in \mathbb{R}$ is a scalar, $I \in \mathbb{R}^{n^\ell \times n^\ell}$ is an identity matrix, and $\hat{P} \in \mathbb{R}^{n\times n}$. Since GEVP~\eqref{eq:prob_gevp} is quasi-convex, we can find an approximate global optimum with an arbitrary discrepancy by solving convex problems~\cite[Chapter 4.2.5]{boyd2004convex}. Constraint~\eqref{eq:cond23_lyap} is a stability condition that checks whether $\hat{A}(h)$ is Hurwitz stable for all $h \in \mathcal{H}$. The justification is given in Appendix~\ref{app:con_stb}.



\begin{lem}\label{lem:lpv}
For a given $h^\text{o}_+ > 0$, if $\beta$ and $P$ are the solution of GEVP~\eqref{eq:prob_gevp}, $V(x) = \Delta x^\top P\Delta x$ is a common Lyapunov function for~\eqref{eq:mdl_small_sec} with a negative definite time derivative for all $h \in \mathcal{H}$, where $\mathcal{H} = \{ h: h \in [0,\beta h_{+}^{\text{o}}] \}$.
\end{lem}
Lemma~\ref{lem:lpv} shows the Hurwitz stability of $\hat{A}$ for any realization of $h$ in $\mathcal{H}$. 

\subsection{Nonlinear System Condition}\label{sec:tran_stb:non}
The condition on the LPV system allows us to develop transient stability conditions for the original nonlinear system. Recall that to ensure that the set $\Delta \mathcal{X}^\text{0}$ is a subset of the ROA, we can check the three conditions stated in Section~\ref{sec:pre:roa}. 

\subsubsection{Condition a)} This condition is essentially to check whether or not there exists a positive scalar $\alpha$ such that $\Delta \mathcal{X}^{\text{o}}$ is a subset of a sub-level set, $\mathcal{S}(P,\alpha) = \{\Delta x: \Delta x^\top P \Delta x \leq \alpha\}$. Given $P$ as the solution of GEVP~\eqref{eq:prob_gevp}, the corresponding set-covering problem is formulated as follows: 
\begin{subequations}\label{eq:con_set}
    \begin{alignat}{2}
        \textbf{Set-Covering:}& \quad &&\min_{\alpha > 0} \alpha \\
         \text{s. t.}& &&\left( \Delta \mathcal{X}^{\text{o}}_k \right)^\top P \Delta \mathcal{X}^{\text{o}}_k \leq \alpha,\, k = 1,\cdots\, ,m,\label{eq:cond1_LMI}
    \end{alignat}
\end{subequations}
where $\Delta \mathcal{X}^{\text{o}}_k$ is the $k$-th vertex of $\Delta \mathcal{X}^{\text{o}}$, and $m$ is the number of vertices. Set-Covering problem~\eqref{eq:con_set} is a convex problem with convex quadratic inequalities owing to the positive definiteness of $P$. Moreover, since $\Delta \mathcal{X}^\text{o}$ is a convex polytope with interval constraints, and $\mathcal{S}$ is an ellipsoid, we only need to make sure that all the vertices of $\Delta \mathcal{X}^\text{o}$ are contained in $\mathcal{S}$ to let $\Delta \mathcal{X}^\text{o}$ be a subset of $\mathcal{S}$. 
 
The computational efficiency of~\eqref{eq:con_set} can be significantly improved by leveraging the geometric structure of $\Delta \mathcal{X}^\text{o}$. Noting that $\Delta \mathcal{X}^\text{o}$ involves symmetric interval constraints, we can first find a minimum-volume symmetric ellipsoid to contain this set, and then find a sub-level set to contain the symmetric ellipsoid. Each of the resulting set-covering problems involves only one LMI. Note that the former problem is equivalent to enforcing that an arbitrary vertex of $\Delta \mathcal{X}^\text{o}$ is included in the symmetric ellipsoid, and the latter problem involves a well-known ellipsoid-covering problem~\cite{boyd2004convex}.

Moreover, we formulate a co-design problem that not only combines the aforementioned two set-covering LMI problems but also finds $\mathcal{H}$ and $\mathcal{S}$ together:

\begin{subequations}\label{eq:prob_codesign}
    \begin{alignat}{2}
        \textbf{Co-Design:}& \quad &&\max_{P\succ 0,\hat{P},\gamma > 0,\tau > 0} \det\_{\mathrm{rootn}} (\hat{P})\label{eq:prob_codesign_cost} \\
         \text{s. t.}& &&\left( \Delta \mathcal{X}^{\text{o}}_k \right)^\top [\gamma] \Delta \mathcal{X}^{\text{o}}_k \leq 1, \; P \preceq [\gamma],\label{eq:prob_codesign:con1}\\
         & && ~\eqref{eq:cond23_lyap},~\eqref{eq:cond23_misc},
    \end{alignat}
\end{subequations}

\hspace{-0.1in}where $k$ is an arbitrary integer in $[1,m]$, and $\gamma \in \mathbb{R}^{n}_+$. Constraint~\eqref{eq:prob_codesign:con1} involves the two set-covering LMIs: 1) we check whether an arbitrary vertex of~$\Delta \mathcal{X}^\text{o}$ is contained in a symmetric ellipsoid defined by~$\{\Delta x: \Delta x^\top [\gamma] \Delta x \leq 1\}$; 
and 2) the constraint $P \preceq [\gamma]$ ensures that the symmetric ellipsoid is a subset of a sub-level set $\mathcal{S}(P,1) = \{\Delta x: \Delta x^\top P \Delta x \leq 1\}$. 
The cost function~\eqref{eq:prob_codesign_cost} seeks to minimize the volume of the sub-level set $\mathcal{S}(P,1)$. For brevity, the argument of $\mathcal{S}(P,1)$ is dropped in the remainder of the paper.

Unlike GEVP~\eqref{eq:prob_gevp}, Co-Design problem~\eqref{eq:prob_codesign} treats $\beta$ as a parameter, which makes it a convex LMI problem. One can use a line-search method to gradually increase the value of $\beta$ from an initial guess to find the largest $\beta$ that makes~\eqref{eq:prob_codesign} feasible. If~\eqref{eq:prob_codesign} is feasible for some $\beta$, $\mathcal{H}$ and $\mathcal{S}$ can be obtained subsequently. 

\begin{definition}\label{def:hs}
If problem~\eqref{eq:prob_codesign} is feasible for some $\beta > 0$, and $P \succ 0$ is a solution of such problem, we call $\mathcal{H} = \{h:h\in [0,\beta h^\text{o}_+]\}$ the \textit{LPV-stability set}, and $\mathcal{S} = \{\Delta x: \Delta x^\top P \Delta x \leq 1\}$ the \textit{minimum sub-level set}.
\end{definition}


\subsubsection{Condition b)} This condition is trivial to check after the change of coordinates and the choice of the candidate Lyapunov function. 

\subsubsection{Condition c)} As for the last condition, the stability analysis results for the LPV system provides the following intuitive insight: if $V(\Delta x)$ is a common Lyapunov function that guarantees the Hurwitz stability of LPV system matrix $\hat{A}$ so long as the parameter uncertainty $ h \in \mathcal{H}$, and if we can ensure that any $\Delta x \in \mathcal{S}$ keeps the nonlinear function $h \in \mathcal{H}$ for the nonlinear system~\eqref{eq:mdl_small_ini}, then condition c) must satisfy.

From the above intuition, we just need to derive conditions to ensure $h \in \mathcal{H}$, which is given by:
\begin{lem}\label{lem:con34}
Under Assumption~\ref{assum:positive}, given $x^e$, LPV-stability set $\mathcal{H}$, and minimum sub-level set $\mathcal{S}$, $h$ belongs to $\mathcal{H}$ for any $\Delta x \in \mathcal{S}$ if the following constraints hold:
\begin{subequations}\label{eq:con34}
    \begin{align}
    &\sup_{\Delta x \in \mathcal{S}}h \leq \beta h_+^{\text{o}},\label{eq:con34_sup}\\
    &\inf_{\Delta x \in \mathcal{S}} h \geq 0.\label{eq:con34_inf}
\end{align}
\end{subequations}
\end{lem}
Recall that $h$ is a function of $x^e$ and $\Delta x$, and $\sup$ and $\inf$ are element-wise operators. These conditions are further transformed into a linear constraint in $x^e$, which is amenable for control design.

Let $\inf_{\Delta x \in \mathcal{S}}C_1\Delta x$ be denoted by $\Delta x^{C}_{\inf}$ for brevity. 
We first study the feasibility condition of~\eqref{eq:con34}:
\begin{lem}\label{lem:infty}
Under Assumption~\ref{assum:positive}, constraints~\eqref{eq:con34} is feasible for some $x^e$
if and only if there exists $x^e$ that satisfies \mbox{$C_1x^e+\Delta x^{C}_{\inf} > 0$}.
\end{lem}
The justification of Lemma~\ref{lem:infty} is given in Appendix~\ref{app:lem:infty}. If this condition is satisfied, inequality~\eqref{eq:con34_inf} is trivial to check. The implication of Lemma~\ref{lem:infty} is consistent with engineering practice noting that a negative voltage trajectory is not allowed. Under this condition, $\sup_{\Delta x \in \mathcal{S}}h$ is given by:
\begin{equation}
    \sup_{\Delta x \in \mathcal{S}} h = -p_\ell \varoslash \left( \left(C_1  x^e + \Delta x^{C}_{\inf}\right)\odot C_1x^e\right).\label{eq:con34_sup_qd}
\end{equation}
Then~\eqref{eq:con34_sup} is equivalent to the following linear inequality:
\begin{equation}
    C_1x^e \! \geq \! \underbrace{\frac{1}{2} \left( - \Delta x^{C}_{\inf} \! + \! \sqrt{\left[\Delta x^{C}_{\inf}\right]\Delta x^{C}_{\inf} \! - \! 4p_\ell\varoslash \beta h^\text{o}_+} \right)}_{x^{Ce}_-}.\label{eq:con34_sup_ln}
\end{equation}
The derivation is given in Appendix~\ref{app:sqrt}. 
Note that we define a composite parameter $x^{Ce}_-$ for the right-hand-side. It is a constant if $\mathcal{H}$ and $\mathcal{S}$ are given.
\subsubsection{Control-Amenable Stability Condition}
With the above results, we develop the control amenable condition to certify transient stability as follows:
\begin{prop}\label{prop:tran_stb}
Given $x^e$, system~\eqref{eq:mdl_small_sec} is transiently stable with a guaranteed ROA if a) problem~\eqref{eq:prob_codesign} is feasible for some $\beta>0$, and b) inequality~\eqref{eq:con34_sup_ln} is satisfied.
\end{prop}
Since problem~\eqref{eq:prob_codesign} is feasible for some $\beta > 0$, we let $P\succ 0$ be the solution. Thus, the LPV-stability set $\mathcal{H}$ and the minimum sub-level set $\mathcal{S}$ as defined in Definition~\ref{def:hs} are available. This implies that $\Delta \mathcal{X}^\text{o}$ lies in the sub-level set $\mathcal{S}$ of $V(\Delta x)$ owing to the set-covering constraint~\eqref{eq:prob_codesign:con1} (condition a)).

It is clear that satisfying~\eqref{eq:con34_sup_ln} implies that the condition stated in Assumption~\ref{assum:positive} and the feasibility condition derived in Lemma~\ref{lem:infty} are both satisfied. Hence, Lemma~\ref{lem:con34} certifies that for any realization of $\Delta x$ in $\mathcal{S}$ the corresponding $h$ belongs to $\mathcal{H}$.
Owing to the fact that problem~\eqref{eq:prob_codesign} is feasible, GEVP~\eqref{eq:prob_gevp} is feasible as well. Lemma~\ref{lem:lpv} certifies that $\dot{V}(\Delta x)$ is negative definite in $\mathcal{S}$ except at the origin where it is zero (condition c)).

As indicated above, condition b) always holds due to the change of coordinates and the definition of the Lyapunov function. Thus, all three conditions discussed in Section~\ref{sec:pre:roa} are satisfied, which shows that $\Delta \mathcal{X}^\text{o}$ is a subset of the system ROA.

The condition we derive in Proposition~\ref{prop:tran_stb} indicates that we can design the system operating point to ensure transient stability with a guaranteed ROA. The structure of this single linear condition is easy to check and is amenable for control design. 

\subsection{Control Synthesis}\label{sec:tran_stb:opf}
In this subsection, we show how the condition derived in Proposition~\ref{prop:tran_stb} can be utilized to synthesize a transient stabilization control. The design problem resembles OPF~\eqref{eq:prob_opf}, and thus can be efficiently solved by existing solvers.

Note that the condition derived in Proposition~\ref{prop:tran_stb} characterizes a specific area for steady-state system operation such that any point lying within it is transiently stable with a guaranteed ROA. As discussed in Section~\ref{sec:pre:opf}, we can design $u$ to drive the operating point to a desired region. Subsequently, the following control synthesis problem is formulated:
\begin{subequations}\label{eq:prob_stb_opf}
\begin{align}
    \text{\textbf{Control Synthesis:}} & && \min_{u,x^{\text{e}},v^\text{e}_\ell, p_{\text{s}}} c^\top p_{s} ,\label{eq:opf_tran_stb:cf}\\
     \text{s. t.}& && \eqref{eq:prob_opf:con1},\; \eqref{eq:prob_opf:con2},\;\eqref{eq:con34_sup_ln}.
\end{align}
\end{subequations}

Control synthesis problem~\eqref{eq:prob_stb_opf} seeks to optimize the system operating point to 1) reduce operational costs, 2) satisfy operational constraints, and 3) ensure transient stability with a guaranteed ROA. 
The entire analysis and control design process can be summarized in the following algorithm:
\begin{algorithm}
	\caption{Design control input $u$ to make system~\eqref{eq:mdl_main} transiently stable with a guaranteed ROA}\label{alg:main}
	\textbf{Input:} System matrices: $A$, $B_1$, $B_2$, $D$; constraint sets: $\mathcal{X}^e$, $\mathcal{U}$, and $\mathcal{P}^{\text{s}}$; ROA needed to guarantee: $\Delta \mathcal{X}^\text{o}$; load profile: $p_\ell$; and initial guess: $h^\text{o}_+$ and $\beta$.\\
	\textbf{Output:} Control $u$.
	\begin{algorithmic}[1]
		\item[Step 1:] Obtain LPV-stability set $\mathcal{H}$ and minimum sub-level  
		\item[] \quad \; set $\mathcal{S}$ using a line-search method to find the largest $\beta$
		\item[] \quad \; that makes Co-Design problem~\eqref{eq:prob_codesign} feasible.
		\item[Step 2:] Find $\Delta x^C_{\inf} = \inf_{\Delta x \in \mathcal{S}}C_1\Delta x$.
		\item[Step 3:] Calculate $x^{Ce}_-$ using~\eqref{eq:con34_sup_ln}.
		\item[Step 4:] Find $u$ by solving control synthesis problem~\eqref{eq:prob_stb_opf}.
	\end{algorithmic}
\end{algorithm}

The performance guarantee of Algorithm~\ref{alg:main} is given as follows:
\begin{thm}\label{thm:tran_stb}
Any feasible solution $u$ obtained using Algorithm~\ref{alg:main} ensures transient stability with a guaranteed ROA for system~\eqref{eq:mdl_main}.
\end{thm}
Theorem~\ref{thm:tran_stb} is a direct result from Proposition~\ref{prop:tran_stb}. It can be used directly as it only involves LMI problem Co-Design~\eqref{eq:prob_codesign}, convex problem $\Delta x^C_{\inf} = \inf_{\Delta x \in \mathcal{S}}C\Delta x$, and some basic algebraic calculations, aside from the control synthesis problem~\eqref{eq:prob_stb_opf}. Compared to a conventional OPF~\eqref{eq:prob_opf}, we only introduce one additional linear constraint into~\eqref{eq:prob_stb_opf}. Therefore, the tractability of the proposed problem is similar to that of an OPF, and existing OPF solution algorithms can be applied to our problem as well. Additionally, it is worth mentioning that it suffices to find any $\beta$ that makes Co-Design problem~\eqref{eq:prob_codesign} feasible in Step 1 to ensure the guaranteed ROA, although a line-search method is utilized to find the largest $\beta$ for the purpose of conservativeness reduction.

\subsection{Discussions}\label{sec:tran_stb:ext}
We discuss practical applications and potential theoretical extensions of the proposed work in this subsection.

\subsubsection{Convex Relaxation}
Given the structure of~\eqref{eq:prob_stb_opf}, OPF relaxation methods can be applied to reduce the computational complexity. OPF convex relaxation approaches seek to relax the nonlinear power flow equations into convex counterparts in order to transform a non-convex OPF problem into a convex problem. For example, the OPF problems can be relaxed into second order cone programming and semidefinite programming problems~\cite{molzahn2018survey}. Under some conditions such relaxations are exact for DC microgrids~\cite{gan2014optimal,li2018optimal}. For example, ~\cite[Theorem~2]{li2018optimal} shows that under the condition that the DC bus voltage is subject to a uniform upper bound, the second order conic relaxation is exact. As our method only introduce a new lower bound on voltage constraints, such relaxation method is exact.

\subsubsection{Other Models}
The proposed work can include alternative component models. For example, the sources might have different operating strategies. Distributed control~\cite{morstyn2015unified} and decentralized feed-forward control~\cite{gui2021large} have been developed for DC microgrids. In addition, some sources may act as constant power sources (CPSs), like a PV unit operated at the maximum power point tracking mode, to inject power into the network. Our approach can be applied to these sources as well. Regarding most DC microgrid controllers, their closed-loop dynamics conform to the general structure of model~\eqref{eq:mdl_main}, and most of them need to design a voltage setpoint. Hence, our works can be extended for their design. Additionally, the inclusion of CPSs does not change the structure of the system model. Instead, only minor modifications to the proposed work are required. For example, we need to define a new lower bound of $h$ similar to $\beta h^{\text{o}}_+$ in $\mathcal{H}$. Our methodology can be easily extended to address such changes.

\section{Case Studies}\label{sec:simu}
In this section, we show the validity and relevance of the proposed work using simulations on realistic networks. The simulations are conducted on a laptop with an Intel Core i5-1.7GHz 8-core CPU and 16-GB of RAM. Algorithm~\ref{alg:main} is executed using CVX~\cite{cvx} and IPOPT~\cite{wachter2006implementation}. Dynamic simulations are conducted in MATLAB/Simulink.

We first use a one-bus DC microgrid to illustrate the conservativeness of the proposed technique, and then use a 14-bus DC microgrid to show the applicability and computational efficiency of the proposed technique.

\subsection{Conservativeness}\label{sec:onebus}
A one-bus DC microgrid is studied in this subsection to illustrate the conservativeness of the proposed methodology. The microgrid circuit and its parameters are given in Fig.~\ref{fig:rev1_topo} and Table~\ref{table:rev1_para}, respectively. The system matrices are provided in Appendix~\ref{app:mdl}. As shown in Fig.~\ref{fig:rev1_topo}, the microgrid has a voltage source supporting a CPL. The current and the DC bus voltage are the two state variables. We wish to ensure that the system is transiently stable with a maximum voltage and current fluctuation of 20 V and 20 A, respectively, from the operating point. It is worth mentioning that for the special case in which the system has only one source, the cost function of problem~\eqref{eq:prob_stb_opf} is meaningless and, therefore, is modified to minimize the value of $u$. Additionally, we make the inequality constraints nonbinding to prevent the steady-state performance requirements to impair the assessment of the transient performance.
    \begin{figure}[ht!]
    \centering
    \includegraphics[width=0.3\textwidth]{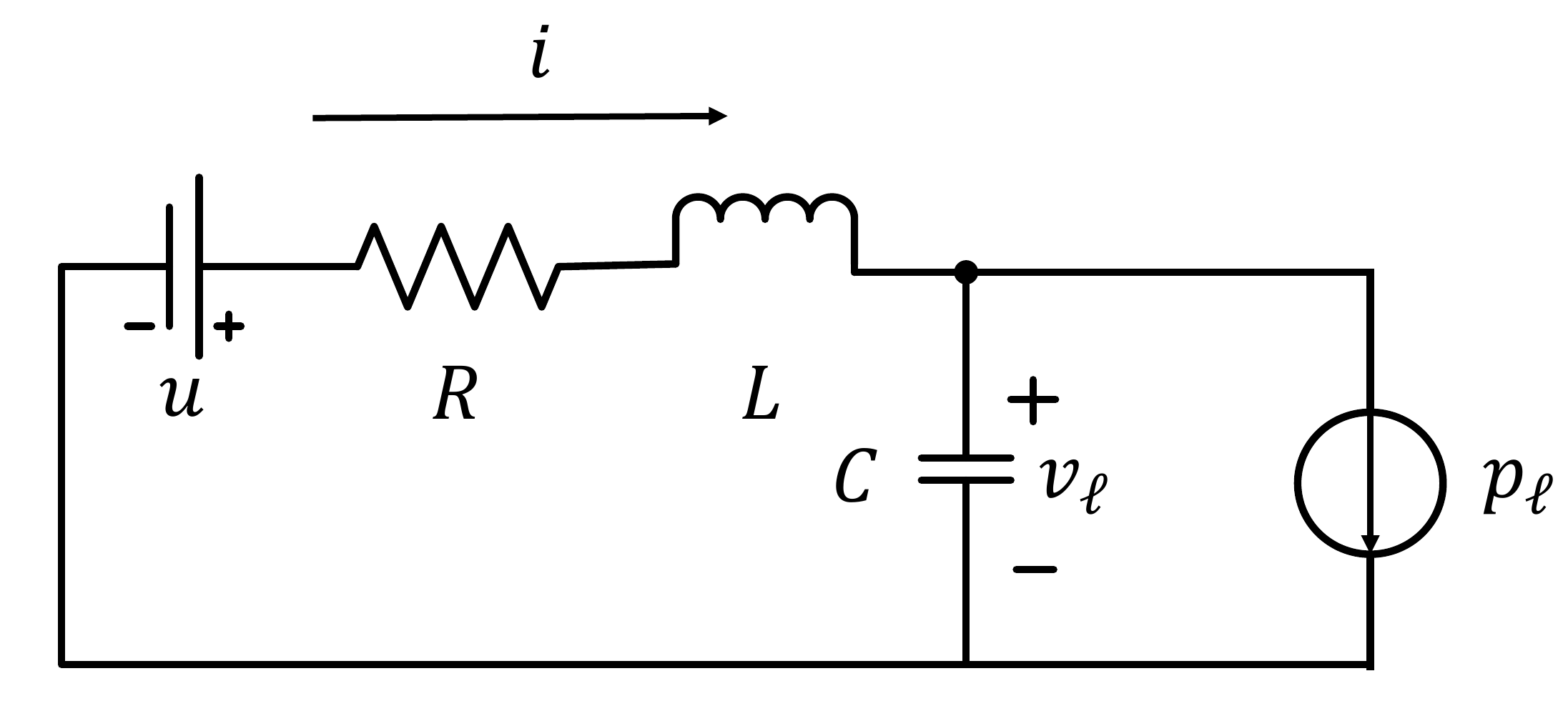}
    \caption{Topology of the one-bus DC microgrid}\label{fig:rev1_topo}
    \end{figure} 
    
    \begin{table}[!ht]
	\caption{Parameters of the one-bus DC microgrid}
	\label{table:rev1_para}
	\begin{center}
		\begin{tabular}{cccccccc}
			\hline\hline
			$R$ & 0.5 $\Omega$ & $L$ & 1 mH & $C$ & 0.5 mF & $p_\ell$ & 300 W\\
			\hline\hline
		\end{tabular}
	\end{center}
    \end{table}

Our analysis finds that the control objective can be achieved if the steady-state voltage is larger than 62.3 V. The control input $u$ is synthesized to be 64.8 V. We use numerical methods to approximately obtain the system ROA and to plot it along with the designated operating range $\Delta \mathcal{X}^\text{o}$ and the minimum sub-level set $\mathcal{S}$ in Fig.~\ref{fig:rev1_ROA}. The following observations are in order: 1) The proposed methodology can successfully ensure the inclusion of the entirety of the designated operating range in the ROA, 2) the lower left corner of the designated operating range is close to the ROA boundary, and 3) the ellipse is visually close to a maximum-volume ellipse centering at the operating point that not only covers the designated operating range but is also a subset of the system ROA. The last two observations qualitatively show that the proposed methodology is reasonably conservative. 

    \begin{figure}[ht!]
    \centering
    \begin{subfigure}[b]{0.4\textwidth}
        \centering
        \includegraphics[width=\textwidth]{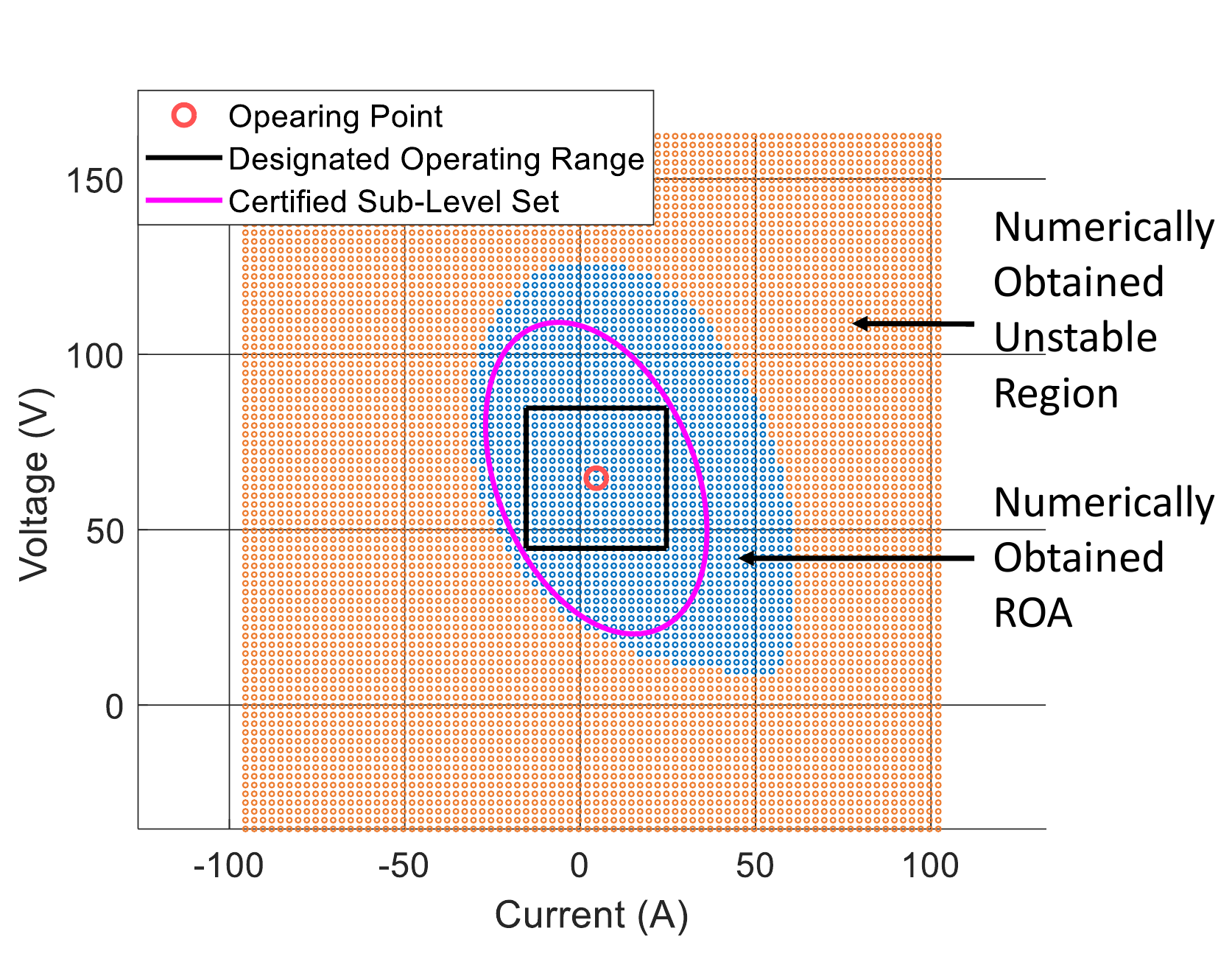}
        \caption{Numerically obtained ROA for one-bus DC microgrid}
    \end{subfigure}
    \hfill
    \begin{subfigure}[b]{0.35\textwidth}
        \centering
        \includegraphics[width=\textwidth]{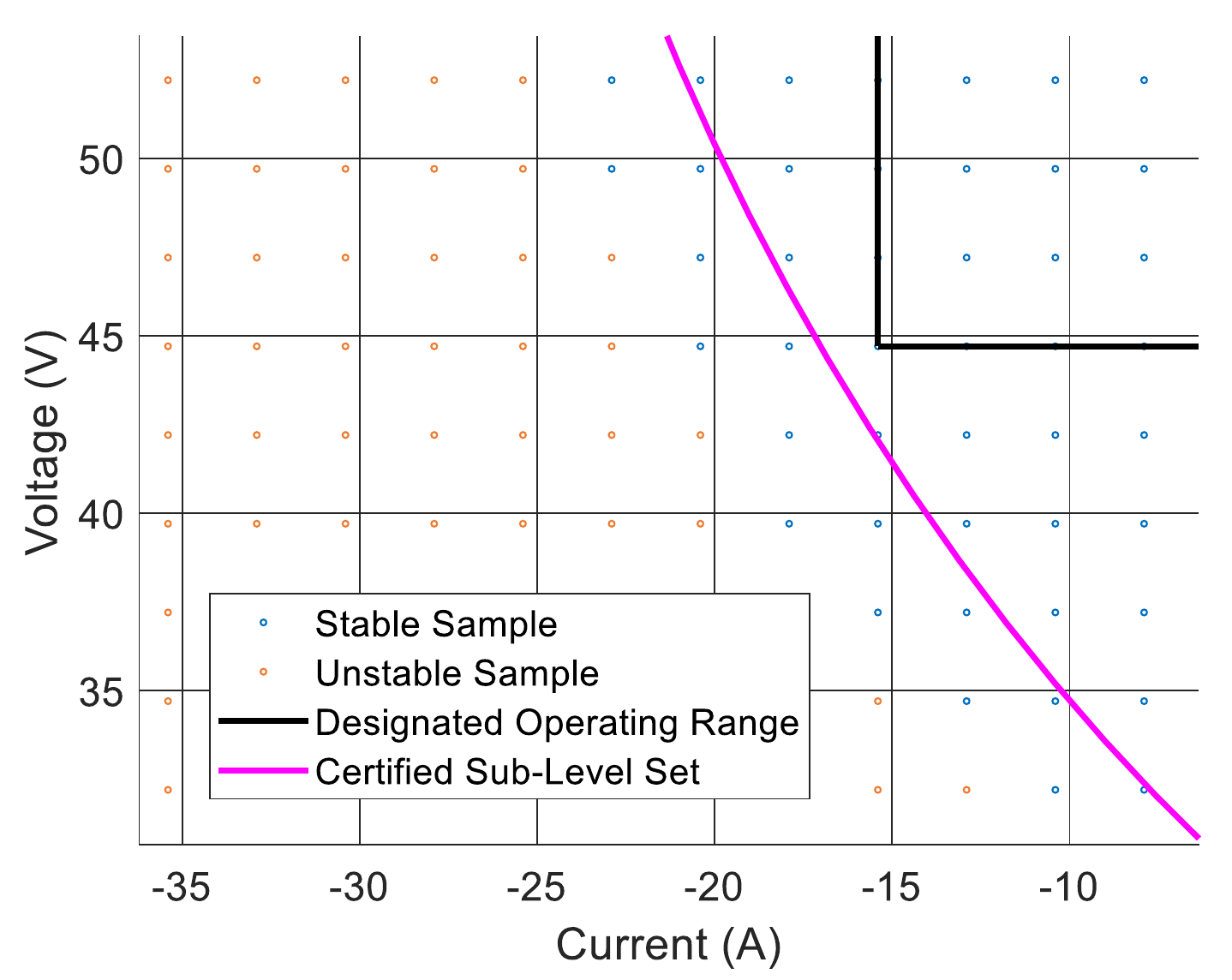}
        \caption{Zoomed in view on lower left boundary}
    \end{subfigure}
    \caption{Numerically obtained DC microgrid ROA}\label{fig:rev1_ROA}
    \end{figure} 
    
We compare the synthesized control with the borderline design, which is the smallest control input that achieves the control objective. Relative differences (RDs) between the synthesized control and the borderline design are used to quantify the conservativeness\footnote{ RD is the ratio of the absolute difference to the borderline design.}. The borderline design is obtained by numerical methods as 59.4 V, which yields a 0.09 RD.

We apply the same methods to a group of scenarios with different designated operating ranges. The results are reported in Table~\ref{table:rev1_sum1}. It can be observed that the quantified conservativeness is consistent across these scenarios.

\subsection{Sensitivity Regarding Parameters}
We further study in this subsection the sensitivity of the proposed methodology with respect to changes in system parameters. We apply the proposed methodology to the one-bus DC microgrid with four groups of scenarios: 1) varying CPL, 2) varying constant impedance loads (CZL) connected in parallel to the CPL, 3) varying inductance, and 4) varying capacitance. The results are given in Tables~\ref{table:rev1_sum4} --~\ref{table:rev1_sum7}. Note that in each group (except for the cases with the superscript ``*'') we only change one parameter relative to those in Table~\ref{table:rev1_para}. Regarding the group with varying CPL, we can observe from Table~\ref{table:rev1_sum4} that the quantified conservativeness is consistent, which shows that our methodology is not sensitive to CPL demands. In Table~\ref{table:rev1_sum5}, the scenario with an $\infty$ CZL represents the base case with only CPL, while in the scenario with 1 $\Omega^*$ CZL, we let the CPL increase to 1 kW from 300 W. We find that with 300 W CPL, the quantified conservativeness drastically increases when the CZL is 1 $\Omega$, but is consistent in other cases. Nonetheless, when the CPL is increased to 1 kW, the quantified conservativeness becomes comparatively consistent. With 300 W CPL, the power consumed by the CZL is around 2.1 kW, whereas it is about 3 kW in the CZL 1 $\Omega$\textsuperscript{*} case. One insight is that if the power consumed by the CZL is dominant over the CPL, the conservativeness of the proposed methodology might increase. From Tables~\ref{table:rev1_sum6} and~\ref{table:rev1_sum7}, we find that the increase of inductance or the decrease of capacitance generally results in increasing the conservativeness of the proposed work. Especially, with 10 mH inductance or 80 $\mu$F capacitance, the conservativeness becomes significantly higher than in other cases. Nonetheless, if in scenarios with the superscript ``*'' we modify the composite resistance to 1 $\Omega$ by increasing the droop gain, the conservativeness returns to a comparatively consistent level. This observation illustrates the potential benefits of designing droop gains along with the droop references. Note that we can make the resistance an additional design variable at the expense that the corresponding optimization problems, such as~\eqref{eq:prob_codesign}, will involve non-convex bilinear terms. Nevertheless, there exists effective methods to tackle bilinear terms~\cite{hassibi1999path,vanantwerp2000tutorial}.
    
    \begin{table}[!ht]
	\caption{Sensitivity study on operating ranges} 
	\label{table:rev1_sum1}
	\begin{center}
		\begin{tabular}{cccc}
			\hline\hline
			Range & (20 A, 20 V) & (10 A, 20 V) & (20 A, 10 V)\\
			Control & 64.8 V & 56.2 V  & 59.7 V \\
			RD & 0.09 & 0.11 & 0.11 \\
			\hline\hline 
		\end{tabular}
	\end{center}
    \end{table}
    
    
    
    \begin{table}[!ht]
	\caption{Sensitivity study on CPL} 
	\label{table:rev1_sum4}
	\begin{center}
		\begin{tabular}{ccccc}
			\hline\hline 
		    CPL	    & 0.3 kW   & 3 kW    & 5 kW    & 10 kW\\
			Control & 64.8 V  & 145.9 V & 181.8 V & 248.1 V\\
			RD      & 0.09    & 0.12    & 0.11    & 0.10\\ 
			\hline \hline\\
		\end{tabular}
	\end{center}
    \end{table}
    
    \begin{table}[!ht]
	\caption{Sensitivity study on CZL}
	\label{table:rev1_sum5}
	\begin{center}
		\begin{tabular}{ccccc}
			\hline\hline 
		    CZL	    & $\infty$   & 100 $\Omega$    & 1 $\Omega$    & 1 $\Omega^*$\\
			Control & 64.8 V  & 64.2 V & 72.8 V & 92.7 V\\
			RD      & 0.09    & 0.08    & 0.27    & 0.11\\ 
			\hline \hline\\
		\end{tabular}
	\end{center}
    \end{table}
    
    \begin{table}[!ht]
	\caption{Sensitivity study on inductance} 
	\label{table:rev1_sum6}
	\begin{center}
		\begin{tabular}{ccccc}
			\hline\hline 
		    $L$	    & 1 mH   & 3 mH      & 10 mH    & 10 mH\textsuperscript{*}\\
			Control & 64.8 V & 97.7 V    & 170.4 V  & 144.1 V\\
			RD      & 0.09   & 0.15      & 0.21     & 0.12 \\ 
			\hline \hline\\
		\end{tabular}
	\end{center}
    \end{table}
    
    \begin{table}[!ht]
	\caption{Sensitivity study on capacitance} 
	\label{table:rev1_sum7}
	\begin{center}
		\begin{tabular}{ccccc}
			\hline\hline 
		    $C$	    & 0.5 mF    & 0.3 mF   & 80 $\mu$F  & 80$\mu$F\textsuperscript{*}\\
			Control & 64.8 V    & 77.2 V   & 135.3 V    & 117.7 V\\
			RD      & 0.09      & 0.12     & 0.19       & 0.11\\ 
			\hline \hline\\
		\end{tabular}
	\end{center}
    \end{table}

\subsection{Case Studies on 14-Bus DC Microgrid}
We then study a 14-bus DC microgrid, as shown in Fig.~\ref{fig:ieee14}. It is analogous to the IEEE 14-bus system with a meshed network~\cite{ieee14}. The generators are all voltage sources operated under primary droop control, and we use a composite resistance to equivalently model the droop gains combined with the physical resistance. Buses 1 and 8 are generator buses, buses 2, 3, and 6 have mixed loads and generators, bus 7 has no loads or generators, and the rest are all load buses; all the loads are CPLs, and the nominal power demand is 10 kW except for those loads at buses 2, 3, and 6, where the demand is 5 kW; the impedance of each power line is represented by a resistor and an inductor connected in series, as indicated in Fig.~\ref{fig:dc_inputs}. The detailed system parameters are listed in Table~\ref{table:simu1_spec}. Note that $c_k$ is the generation cost coefficient, and we consider an equal generation cost; $V_0$ and $p_{\ell 0}$ represent the base voltage and power to obtain per unit (p.u.) values, respectively. 

The control input setpoints and the steady-state voltage are both bounded within [0.9, 1.1] p.u.; each power generation is constrained in the range [0, 3] p.u.; the maximum fluctuation from a steady-state operating point is [-0.18, 0.18] p.u. for each voltage variable and [-0.15, 0.15] p.u. for each current variable, respectively.

\begin{figure}[h!]
\centering
\includegraphics[width=0.5\textwidth]{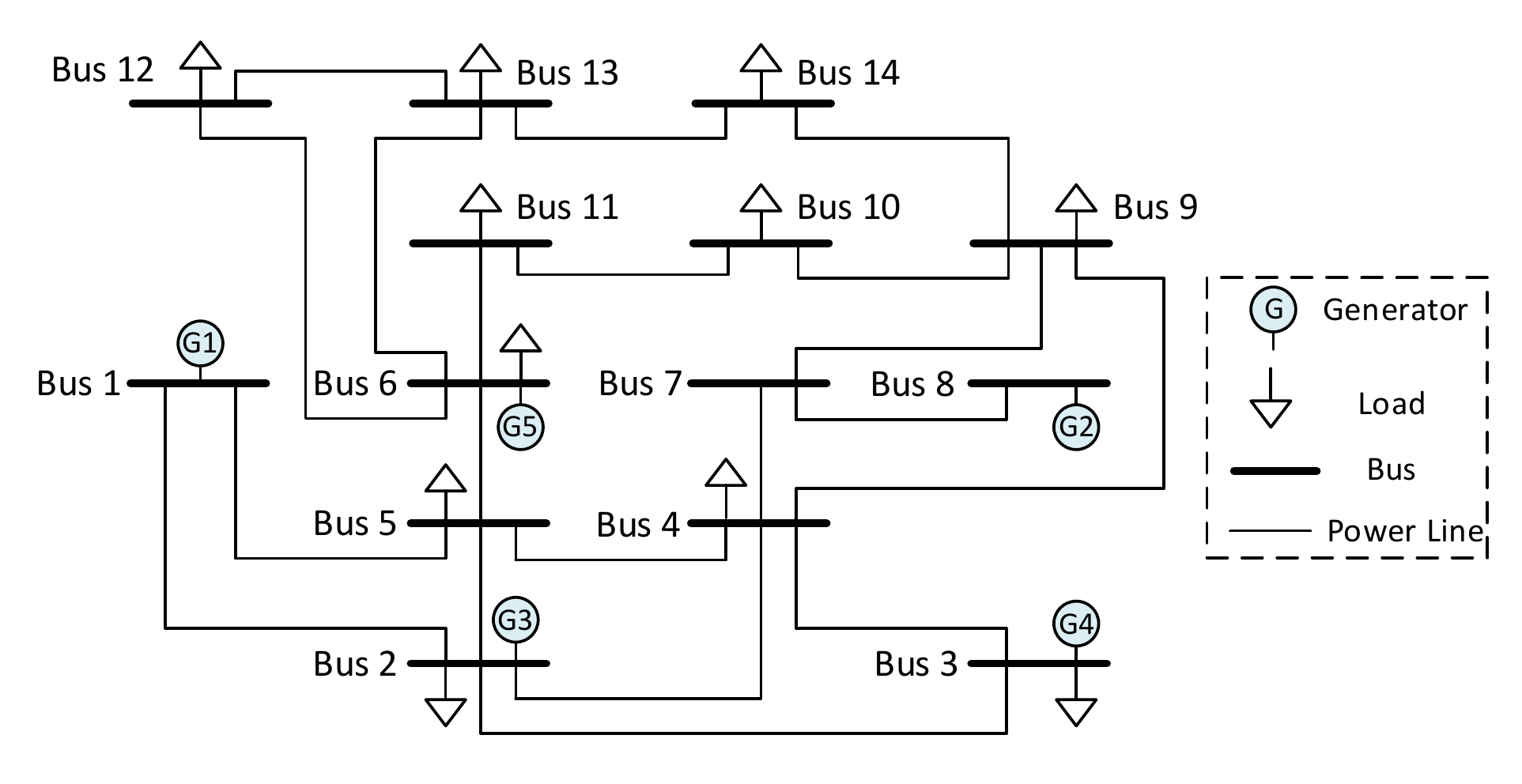}
\caption{Topology of 14-bus DC microgrid}\label{fig:ieee14}
\end{figure} 

\begin{table}[!ht]
	\caption{Parameters of 14-bus DC microgrid} 
	\label{table:simu1_spec}
	\begin{center}
		\begin{tabular}{cccccc}
			\hline\hline
			$R_{sk}$ & 0.1 $\Omega$ & $R_{lj}$ & 5 $\Omega$ & $R_{tp}$ & 0.1 $\Omega$ \\
			$C_{sk}$ & 1 mF & $C_{lj}$ & 1 mF & $L_{tp}$ & 0.8 mH\\
			$V_0$ & 500 V & $p_{\ell 0}$ & 100 kW & $c_k$ & 1 \\
			\hline\hline
		\end{tabular}
	\end{center}
	
\end{table}

We seek to ensure that the system is transiently stable once entering the designated operating range.
To this end, Algorithm~\ref{alg:main} yields voltage setpoints of [1.07, 1.06, 1.06, 1.05, 1.10] p.u. 
We conduct numerical simulations to verify the effectiveness of the proposed methodology. The simulations are initiated from the vertices of the operating range as shown in Fig.~\ref{fig:case0_ini}. Two state variables are selected to plot this two-dimensional figure. Note that ``line 2-3'' represents the power line connecting buses 2 and 3. The simulated system trajectories all converge to the operating point while starting from the boundaries of the operating range. This verifies the correct behavior of the proposed methodology. The CPU time of Steps 2, 3, and 4 of Algorithm~\ref{alg:main} are 0.27 seconds, 0.003 seconds, and 0.14 seconds, respectively. The average computational time for one line-search step in Step 1 is 2.3 seconds. Note that the computational performance can be further improved when commercial solvers and more powerful computational resources are employed.

As for the conservativeness issue, due to the curse of dimensionality, it is computationally intractable to obtain a numerically approximated ROA for this 14-bus system. Nevertheless, one may expect the quantified conservativeness to increase with the system size. The reason is two-fold. 1) As discussed in Section~\ref{sec:tran_stb:non}, we use one LMI in~\eqref{eq:prob_codesign} to replace $m$ LMIs to achieve computational efficiency, but with an increase of system size such replacement may raise the conservativeness. 2) As discussed in Section~\ref{sec:tran_stb:lpv}, we use a uniform scaling factor $\beta$ to assist in finding a maximum-volume $\mathcal{H}$. Such an approximation method may increase the conservativeness with system size as well. 

\begin{figure}[ht!]
\centering
\includegraphics[width=0.4\textwidth]{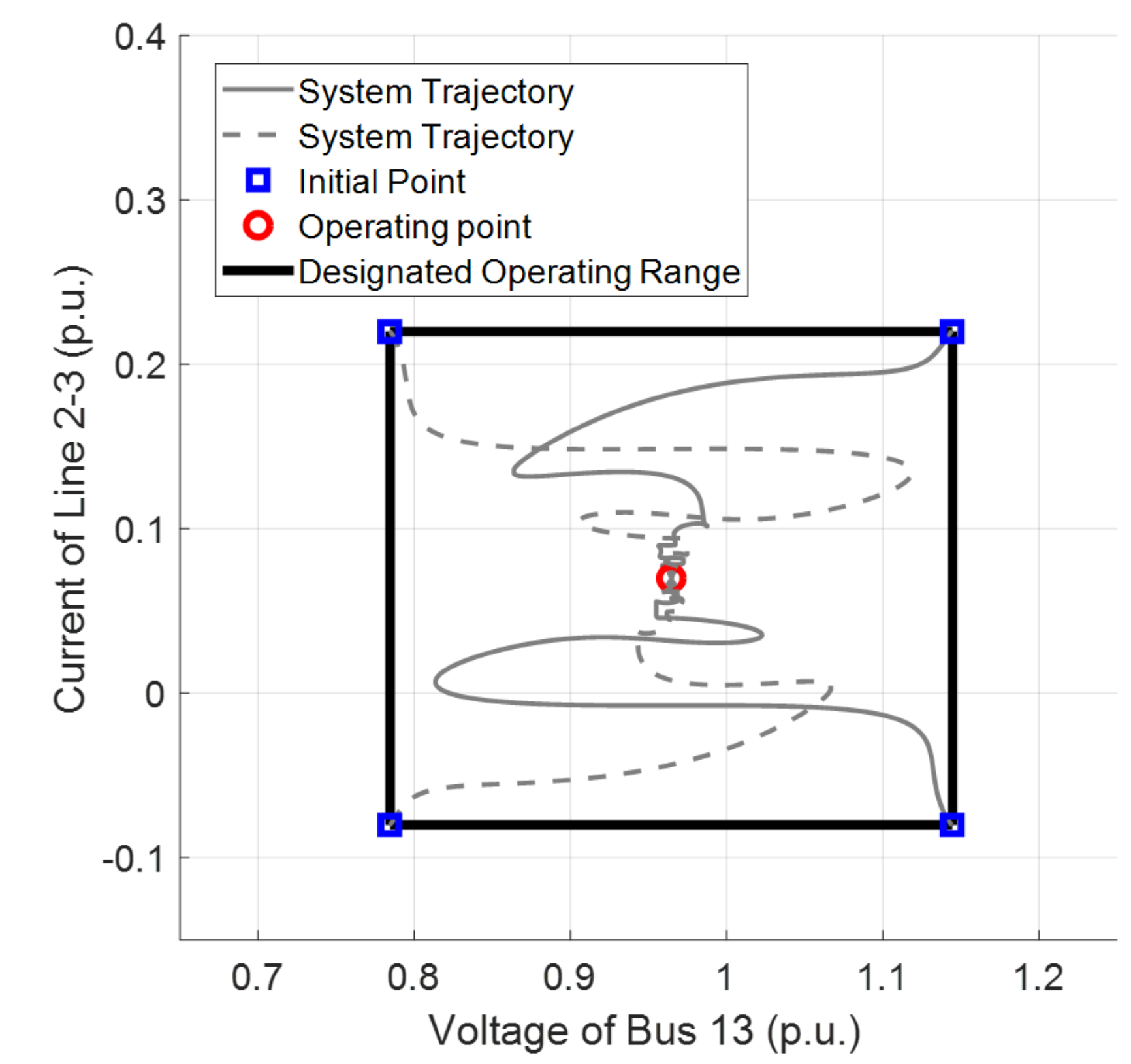}
\caption{Two-dimensional 14-bus DC microgrid trajectories with varying initial conditions}\label{fig:case0_ini}
\end{figure}

\section{Conclusion}~\label{sec:conclu}
This paper develops a novel DC microgrid control synthesis algorithm to guarantee a prescribed region of attraction (ROA). We study the nonlinear dynamics of a general DC microgrid to develop a transient stability condition, which says that an interval set is a subset of the ROA if the system operating point satisfies a linear constraint. A control synthesis problem is then formulated to ensure that this constraint is always satisfied. The problem is structurally similar to a conventional optimal power flow problem, and existing solution algorithms can be applied to solve it efficiently. Simulation case studies show the validity of the proposed work as well as its reasonable conservativeness and computational efficiency. In future work, we will study system parameters' impacts on system transient performance and develop robust control synthesis methods to account for uncertainties in system topology and parameters.

\appendices
\section{Justification for constraint~\eqref{eq:cond23_lyap}}\label{app:con_stb}
\begin{lem}
Matrix $\hat{A}(h)$ is Hurwitz stable for all $h \in \mathcal{H}$ if constraint~\eqref{eq:cond23_lyap} is satisfied for some $P \succ 0$, $\hat{P} = D^{-1}P$, $\tau > 0$ and $T = [\tau]$.
\end{lem}
\begin{proof}
Let $h$ be an arbitrary element in $\mathcal{H}$. Let the voltage deviation of the $k$-th CPL bus be \mbox{$\Delta v_{\ell k}$}. Since $h^\text{o}_{+,k}\geq h_k \geq 0$, the $k$-th entry of $[h]Cx$ is bounded by:
\begin{subequations}
    \begin{alignat}{2}
        &\text{if } \Delta v_{\ell k} < 0, \quad  \beta h_{+,k}^\text{o} \Delta v&&_{\ell k} \leq  h_{k} \Delta v_{\ell k} \leq 0,\nonumber\\
        &\text{if } \Delta v_{\ell k} \geq 0,   &&\; 0 \leq  h_{k} \Delta v_{\ell k} \leq \beta h_{+,k}^\text{o} \Delta v_{\ell k},\nonumber
    \end{alignat}
\end{subequations}
where the expressions of upper and lower bounds alternate according to the sign of $\Delta v_{\ell k}$. Hence, the bounds can be expressed into the following compact inequality:
\begin{equation*}
    \left( h_k \Delta v_{\ell k} - \beta h^\text{o}_{+,k} \Delta v_{\ell k}\right)h_k \Delta v_{\ell k} \leq 0. \nonumber
\end{equation*}
It can be expressed in the vector form as follows:
\begin{align}
    &\left( [h]C_1\Delta x - [\beta h^\text{o}_+]C_1\Delta x \right)^\top [h]C_1\Delta x \leq 0.\nonumber
\end{align}
If $z = [\Delta x^\top; ([h]C_1\Delta x)^\top]^\top$, the equation above can be re-arranged as:
\begin{align}
    z^\top
    \left[ 
    \begin{array}{cc}
        0 & -1/2 \beta C_1^\top [h^\text{o}_+] \\
        -1/2 \beta[h^\text{o}_+]C_1  & I 
    \end{array}
    \right] 
    z \leq 0. \label{eq:s_proc}
\end{align}

For the system matrix $\hat{A}(h)$, we wish to show that
\begin{align}
    &\Delta x^\top \left( P\hat{A}(h)+\hat{A}^\top(h)P + P \right) \Delta x \leq 0, \label{eq:con_stb_main}
\end{align}
given $\hat{A}(h) = D^{-1}\left(A + B_1[h]C_1 \right)$,~\eqref{eq:con_stb_main} is equivalent to:
\begin{align}
    z^\top 
    \left[ 
    \begin{array}{cc}
        PD^{-1}A+A^\top D^{-1}P +P & PD^{-1}B_1 \\
        B_1^\top D^{-1}P  & 0 
    \end{array}
    \right] 
    z \leq 0 \nonumber
\end{align}
from s-procedure, inequality~\eqref{eq:con_stb_main} must hold as there exists $P \succ 0$, $\hat{P} = PD^{-1}$, $\tau > 0$ and $T = \tau I$ that make constraint~\eqref{eq:cond23_lyap} satisfied. The proof is complete.
\end{proof}
\section{Justification of Lemma~\ref{lem:infty}}\label{app:lem:infty}
Let the steady-state voltage of the $k$-th CPL bus be $v^{e}_{\ell k}$ and the voltage deviation be \mbox{$\Delta v_{\ell k}$}. For simplicity, we drop the subscript, $\Delta x \in \mathcal{S}$, of the operators $\sup$ and $\inf$. With the above definitions, $h_k$ is given by,
\begin{equation}
    h_{k} = \frac{-p_{\ell k}}{(v^{e}_{\ell k} + \Delta v_{\ell k})v^{e}_{\ell k}}.\nonumber
\end{equation}

($\Longleftarrow$) Suppose there exists $\hat{v}^{e}_{\ell k}$ that makes \mbox{$\hat{v}^{e}_{\ell k}+\inf \Delta v_{\ell k} > 0$}, any $v^{e}_{\ell k} \geq \hat{v}^{e}_{\ell k}$ makes constraint~\eqref{eq:con34_inf} satisfied. Regarding
\begin{equation}
    \sup h_k = \frac{-p_{\ell k}}{(v^{e}_{\ell k} + \inf \Delta v_{\ell k})v^{e}_{\ell k}},\label{eq:sup_indi}
\end{equation}
inequality $\sup h_k \leq \beta h^{\text{o}}_{+,k}$ is satisfied for any $v^e_{\ell k}$ that satisfy:
\begin{align}
     v^{e}_{\ell k} \geq \frac{-\inf \Delta v_{\ell k} + \sqrt{(\inf \Delta v_{\ell k})^2-4p_{\ell k}/(\beta h^{\text{o}}_{+,k})}}{2}.\label{eq:sqrt_indi}
\end{align}

($\Longrightarrow$) We prove by contradiction. Under Assumption~\ref{assum:positive}, suppose that $v^{e}_{\ell k}+\inf \Delta v_{\ell k} \leq 0$ for any $v^{e}_{\ell k} > 0$, and there exists $v^{e}_{\ell k} > 0$ that makes $\sup h_k \leq \beta h^\text{o}_{+,k}$. 

With the above statement, there exists $\Delta x \in \mathcal{S}$ such that $\Delta v_{\ell k} = -v^{e}_{\ell k}$ for any $v^{e}_{\ell k} > 0$, because $\mathcal{S}$ is a convex ellipsoid. 

As a result, for any finite $v^{e}_{\ell k} > 0$ when $\Delta v_{\ell k} = -v^{e}_{\ell k}$, $h_k = \infty > \beta h^\text{o}_{+,k} \geq \sup h_k$ . This is a contradiction.

\section{Justification of Inequality~\eqref{eq:con34_sup_ln}}\label{app:sqrt}
Inequality~\eqref{eq:con34_sup_qd} is the vector-form representation of~\eqref{eq:sup_indi}, and~\eqref{eq:con34_sup_ln} is the vector form representation of~\eqref{eq:sqrt_indi}. Note that the right-hand-side of~\eqref{eq:sqrt_indi} is the positive solution of the following quadratic equation:
\begin{equation}
    \left( v^{e}_{\ell k} \right)^2 + \inf \Delta v_{\ell k} v^{e}_{\ell k} + \frac{p_k}{\beta h^\text{o}_k} = 0.\nonumber
\end{equation}
This quadratic equation has an negative solution as well. When $v^{e}_{\ell k}$ is smaller than the negative solution or greater than the positive solution, inequality~\eqref{eq:sup_indi} is satisfied. Under Assumption~\ref{assum:positive}, we only need to study the latter case. 

\section{Example System model}\label{app:mdl}
System matrices for the one-bus DC microgrid shown in Section~\ref{sec:onebus} are given as follows:
\begin{subequations}
    \begin{align}
        &A = \left[ 
        \begin{array}{cc}
            -R & -1 \\
            1 & 0
        \end{array}
        \right], B_1 = \left[ 
        \begin{array}{c}
             0  \\
             1 
        \end{array}
        \right], \nonumber \\
        &B_2 = \left[ 
        \begin{array}{c}
             1  \\
             0 
        \end{array}
        \right], D = \left[ 
        \begin{array}{cc}
            L &  \\
             & C
        \end{array}
        \right].\nonumber
    \end{align}
\end{subequations}

\bibliographystyle{IEEEtran}
\bibliography{DCtransient}

\end{document}